\documentclass[11pt,a4paper,reqno]{amsart}
\usepackage{graphicx}
\usepackage[linktocpage=true,colorlinks,citecolor=blue,linkcolor=blue,urlcolor=blue]{hyperref}
\usepackage{amssymb}
\usepackage{cite}
\usepackage{amsmath}
\usepackage{latexsym}
\usepackage{amscd}
\usepackage{amsthm}
\usepackage{mathrsfs}
\usepackage{url}
\usepackage[utf8]{inputenc}
\usepackage[english]{babel}
\usepackage{amsfonts}

\numberwithin{equation}{section}
\def\R{\mathbb R}
\def\Z{\mathbb Z}

\def\d{\mathrm d}
\def\e{\mathrm e}

\def\CM{\mathcal M}
\def\FJ{\mathfrak J}
\def\FG{\mathfrak G}

\def\ee{\varepsilon}
\def\wt{\widetilde}

\def\major{\mathfrak M}
\def\minor{\mathfrak m}

\newtheorem{theorem}{Theorem}[section]
\newtheorem{lemma}[theorem]{Lemma}
\newtheorem{proposition}[theorem]{Proposition}

\theoremstyle{remark}

\theoremstyle{definition}

\theoremstyle{remark}

\numberwithin{equation}{section}

\begin{document}
\title[Primes in short intervals and polynomial phases]{Discorrelation between primes in short intervals and polynomial phases}

\author{Kaisa Matom\"aki}
\address{Department of Mathematics and Statistics, University of Turku, 20014 Turku, Finland}
\email{ksmato@utu.fi}
\thanks{K.M.~was supported by the Academy of Finland grant no. 285894.}

\author{Xuancheng Shao}
\address{Department of Mathematics, University of Kentucky, Lexington, KY, 40506, USA}
\email{xuancheng.shao@uky.edu}
\thanks{X.S.~was supported by the NSF grant DMS-1802224.}

\begin{abstract}
Let $H = N^{\theta}, \theta > 2/3$ and $k \geq 1$. We obtain estimates for the following exponential sum over primes in short intervals:
\[ \sum_{N < n \leq N+H} \Lambda(n) \e(g(n)), \]
where $g$ is a polynomial of degree $k$.  As a consequence of this in the special case $g(n) = \alpha n^k$, we deduce a short interval version of the Waring-Goldbach problem. 
\end{abstract}

\maketitle

\section{Introduction}
Let $N \geq 2$ be a positive integer, and let $H = N^{\theta}$ for some $0 < \theta \leq 1$. The purpose of this paper is to obtain estimates for the sum
\begin{equation}\label{eq:sum} 
\sum_{N < n \leq N+H} \Lambda(n) \psi(n), 
\end{equation}
where $\Lambda$ is the von Mangoldt function, and $\psi$ is a polynomial phase of the form $\psi(n) = \e(g(n))$ (with the notation $\e(x) = \exp(2\pi ix)$) for some polynomial $g$. We would like to obtain results with $\theta$ as small as possible.

In the case of summing over a long interval (i.e. $\theta = 1$), the task of estimating~\eqref{eq:sum} is well understood. When $\deg g = 0$, asymptotic formula for~\eqref{eq:sum} is given by the Prime Number Theorem. When $\deg g = 1$,  estimates for the exponential sum
\[ \sum_{N < n \leq 2N} \Lambda(n) \e(\alpha n) \]
for $\alpha \in \R$ were obtained and used by Vinogradov to solve the ternary Goldbach problem. More generally, for $\psi$ a fixed nilsequence (which includes polynomial phases as special examples), Green and Tao~\cite{GT12mobius} showed the discorrelation estimate
\[ \sum_{N < n \leq 2N} \mu(n) \psi(n) \ll_A \frac{N}{(\log N)^A}, \]
for  any $A \geq 2$. This leads to a discorrelation estimate for~\eqref{eq:sum}, when $\psi$ is in ``minor arc" or when $\Lambda$ is ``W-tricked" (see~\cite[Proposition 10.2]{GT10}).

In the case of summing over a short interval, the case $\deg g = 0$ corresponds to the classical problem of counting primes in short intervals. Huxley's zero density estimate~\cite{Hux72} implies an asymptotic formula for primes in short intervals when $\theta > 7/12$ (see the discussion in~\cite[Chapter 10]{IK04}). When $\deg g = 1$, \eqref{eq:sum} becomes the exponential sum estimate
\[ \sum_{N < n \leq N+H} \Lambda(n) \e(\alpha n). \]
This has been studied quite extensively due to its implication on Vinogradov's theorem with almost equal summands. The best threshold for $\theta$ in this problem is $\theta > 5/8$ due to Zhan~\cite{Zha91}. In the more general case when $g(n) = \alpha n^k$ is a monomial of degree $k$, Huang~\cite[Theorem 2]{Hua16} obtained estimates for~\eqref{eq:sum} when $\theta > 19/24$, and this range is relaxed to $\theta > 3/4$ if $\alpha$ lies in ``minor arcs" (see~\cite[Theorem 1]{Hua16}). When dealing with $\mu$ instead of $\Lambda$, Huang~\cite{Hua15} obtained the estimate
\[ \sum_{N < n \leq N+H} \mu(n) \e(\alpha n^k) \ll_{A} \frac{H}{(\log N)^A} \]
for any $\alpha \in [0,1]$ and $A \geq 2$, in the region $\theta > 3/4$. In this paper we extend the range to $\theta > 2/3$.

\begin{theorem}\label{thm:exp-alpha}
Let $H = N^{\theta}$ for some fixed $\theta > 2/3$. Let $\alpha \in \R$ and let $k$ be a positive integer. Suppose that
\[ \left|\sum_{N < n \leq N +H} \Lambda(n) \e(\alpha n^k) \right|  \geq \frac{H}{(\log N)^A} \]
for some $A \geq 2$. Then there exists a positive integer $q \leq (\log N)^{O_k(A)}$ such that
\[ \|q\alpha\| \leq \frac{(\log N)^{O_k(A)}}{N^{k-1}H}. \]
\end{theorem}

Note that if $q \approx 1$ and $\|q\alpha\| \approx 1/(N^{k-1}H)$, then the phase $\alpha n^k$ is almost constant on $(N, N+H]$ after dividing it into residue classes modulo $q$. This major arc case will thus correspond to the classical prime number theorem in short intervals (in residue classes modulo $q$).

Via the circle method, Theorem~\ref{thm:exp-alpha} leads to a short interval version of the Waring-Goldbach problem. For a prime $p$ and a positive integer $k$, let $\tau = \tau(k,p)$ be the largest integer such that $p^{\tau} \mid k$. Define
\[ \gamma(k,p) = \begin{cases} \tau + 2 & \text{if }p=2\text{ and }\tau > 0, \\ \tau+1 & \text{otherwise.} \end{cases} \]
Define $R(k) = \prod p^{\gamma(k,p)}$, where the product is taken over all primes $p$ with $(p-1) \mid k$.

\begin{theorem}\label{thm:waring-goldbach}
Fix $k \geq 2$, $s \geq k(k+1)+3$, and $\theta > 2/3$. Then every sufficiently large positive integer $N \equiv s\pmod{R(k)}$ can be written as
\[ N = p_1^k + \ldots + p_s^k, \]
where $p_1,\cdots, p_s$ are primes satisfying $|p_i - (N/s)^{1/k}| \leq N^{\theta/k}$.
\end{theorem}

This was proved for $\theta > 19/24$ and $s \geq \max(7, 2k(k-1)+1)$ by Huang~\cite{Hua15}. We refer the reader to~\cite{Hua15} for the historical development of this problem. The improvement on the threshold for $\theta$ comes from Theorem~\ref{thm:exp-alpha}, whereas the improvement on the number of variables $s$ is due to the recent resolution of the main conjecture in Vinogradov's mean value theorem~\cite{BDG16} (which was unavailable to previous authors). Indeed, given Vinogradov's mean value conjecture, Huang's result would require $\theta > 19/24$ and $s \geq k(k+1)+1$. Unfortunately Theorem~\ref{thm:waring-goldbach} is worse in the $s$ respect. Simultaneously to our work, Salmensuu~\cite{Salmensuu19} has applied the transference principle (building on~\cite{MMS17}) to obtain Theorem~\ref{thm:waring-goldbach} for significantly shorter intervals but his work does not provide new information about the exponential sum~\eqref{eq:sum}.

Our original motivation for studying~\eqref{eq:sum} is, in fact, to obtain the short interval version of the aforementioned Green-Tao theorem on discorrelation between primes and nilsequences. We are unable to get any results for general nilsequences $\psi$, but the following theorem deals with the case $\psi(n) = \e(g(n))$ for a general polynomial $g$.

\begin{theorem}\label{thm:exp-g}
Let $k$ be a positive integer. Let $H = N^{\theta}$ for some fixed $\theta > 2/3$. Let $g$ be a polynomial of degree $k$ of the form
\[ g(n) = \sum_{j=1}^k \alpha_j(n-N)^j \]
for some $\alpha_1,\cdots,\alpha_k \in \R$. Suppose that
\[ \left|\sum_{N < n \leq N +H} \Lambda(n) \e(g(n))\right| \geq \frac{H}{(\log N)^A} \]
for some $A \geq 2$. Then there exists a positive integer $q \leq (\log N)^{O_k(A)}$ such that
\[ \|q \alpha_j\| \leq \frac{(\log N)^{O_k(A)}}{H^j} \]
for all $1 \leq j \leq k$.
\end{theorem}

Again note that if $q \approx 1$ and $\|q\alpha_j\| \approx H^{-j}$ for all $j$, then $g(n)$ is almost constant on $(N, N+H]$ after dividing it into residue classes modulo $q$. This is once again a major arc case corresponding to the classical prime number theorem in short intervals.

Unsurprisingly, our argument leads to the following analogous result for the M\"{o}bius function.

\begin{theorem}\label{thm:exp-g-mu}
Let $k$ be a positive integer. Let $H = N^{\theta}$ for some fixed $\theta > 2/3$. Let $g$ be a polynomial of degree $k$. Then
\[ \sum_{N < n \leq N+H} \mu(n) \e(g(n)) \ll_{k,A} \frac{H}{(\log N)^A} \]
for any $A \geq 2$.
\end{theorem}

We end the introduction by mentioning a few related results. In this paper we focus on a fixed short interval, but one can also ask the same question for almost all short intervals. For example, Huxley's zero density estimate implies that one can count primes in almost all short intervals of length $H = N^{\theta}$ with $\theta > 1/6$. In this direction, Matom\"{a}ki and Radziwi{\l}{\l}~\cite{MR16} made the breakthrough showing that
\[ \sum_{n_0 < n \leq n_0+H} \mu(n) = o(H) \]
for almost all $n_0 \sim N$, provided that $H = H(N) \rightarrow \infty$. For the degree $1$ case involving exponential sums, Matom\"{a}ki, Radziwi{\l}{\l}, and Tao~\cite{MRT18} showed that
\[ \sup_{\alpha \in \R} \left|\sum_{n_0 < n \leq n_0+H} \mu(n) \e(\alpha n)\right| = o(H) \]
for almost all $n_0 \sim N$, provided that $H = N^{\theta}$ for any fixed $\theta > 0$, and they plan to return to the higher order cases. Unfortunately these results do not apply to $\Lambda$.

The rest of this paper is organized as follows. In Section~\ref{sec:overview} we outline the general structure of our argument. We also explain there the lack of results for general nilsequences $\psi$. 
In Section~\ref{sec:typeIandII} we prove ``type-I'' and ``minor arc type-II'' exponential sum estimates and in Section~\ref{sec:combine} we prove Theorems~\ref{thm:exp-alpha},~\ref{thm:exp-g}, and~\ref{thm:exp-g-mu}. In Section~\ref{sec:app}, we deduce Theorem~\ref{thm:waring-goldbach} via the circle method. 

\subsection*{Acknowledgments}

We are grateful  for the annonymous referees for valuable suggestions.

\section{Overview of proof}\label{sec:overview}

From now on we will always write $H = N^{\theta}$ and all implied constants are allowed to depend on the degree $k$. We will use $m \sim M$ to denote the dyadic range $M < m \leq 2M$. The proof of Theorem~\ref{thm:exp-g} will begin by an application of Heath-Brown's identity which roughly reduces matters to studying type-I sums
\[
\sum_{\substack{\ell, m \\ m \sim M \\ N < \ell m \leq N + H}} b_m \e(g(\ell m))
\]
and type-II sums 
\[
\sum_{\substack{\ell, m_1, m_2 \\ m_1 \sim M_1, m_2 \sim M_2 \\ N < \ell m_1 m_2 \leq N + H}} a_{\ell} b_{m_1} c_{m_2} \e(g(\ell m_1 m_2))
\]
with $M, M_1, M_2$ on certain ranges. 

Type-I sums will be handled using the following proposition. Here $\tau_{\ell}$ denotes the $\ell$-fold divisor function.

\begin{proposition}[Type-I estimate]\label{prop:typeI}
Let $\delta \in (0,1/2)$ and assume that $1 \leq M \leq \delta^C H$ for some sufficiently large constant $C = C(k) > 0$. Let $|b_m| \leq \tau_5(m)$ for each $m$.  Let
\[ g(n) = \sum_{j=1}^k \alpha_j (n-N)^j \]
be a polynomial of degree $k$. If 
\[ \left|\sum_{\substack{\ell, m \\ m \sim M \\ N < \ell m \leq N+H}} \psi(\ell) b_m \e(g(\ell m))\right| \geq \delta H (\log N)^{13} \]
for either $\psi(\ell) = 1$ or $\psi(\ell) = \log\ell$, then there exists a positive integer $q \leq \delta^{-O_k(1)}$ such that
\[ \|q \alpha_j\| \leq \delta^{-O_k(1)} \frac{1}{H^j} \]
for all $1 \leq j \leq k$.
\end{proposition}

For type-II sums we will have a two-part argument, starting with the following minor arc proposition.
\begin{proposition}[Type-II minor arc estimate]\label{prop:typeII}
Let $\delta \in (0,1/2)$, $M \geq 2$ and $L = N/M$. Assume that $H \geq \delta^{-C}\max(L,M)$ for some sufficiently large constant $C = C(k)>0$. Let $|a_{\ell}| \leq \tau_5(\ell)$ and $|b_m| \leq \tau_3(m)$ for each $\ell$ and $m$. Let
\[ g(n) = \sum_{j=1}^k \alpha_j (n-N)^j \]
be a polynomial of degree $k$. If
\[ \left|\sum_{\substack{\ell, m \\ m \sim M \\ N < \ell m \leq N+H}} a_{\ell} b_m \e(g(\ell m)) \right| \geq \delta H (\log N)^{32}, \]
then there exists a positive integer $q \leq \delta^{-O_k(1)}$ such that
\[ \|q (j \alpha_j + (j+1)N \alpha_{j+1})\| \leq \delta^{-O_k(1)} \frac{N}{H^{j+1}} \]
for all $1 \leq j \leq k$, with the convention that $\alpha_{k+1} = 0$.
\end{proposition}

Propositions~\ref{prop:typeI} and~\ref{prop:typeII} will be proven in Section~\ref{sec:typeIandII}. The Diophantine information in the conclusion of Proposition~\ref{prop:typeII} is perhaps unexpected, but from the argument in Section~\ref{sec:typeIandII} one can see that if
\[ \|j\alpha_j + (j+1)N\alpha_{j+1}\| \approx \frac{N}{H^{j+1}} \]
for all $j$, then the type-II sum could be large:
\[ \sum_{\substack{\ell, m \\ N < \ell m \leq N+H}} a_{\ell} b_m \e(g(\ell m)) \approx H \]
for certain coefficients $\{a_{\ell}\}$ and $\{b_m\}$. In other words, the conclusion in Proposition~\ref{prop:typeII} is the best one can get for general type II sums.

When the coefficients of $g$ satisfy the conclusion of Proposition~\ref{prop:typeII}, we can show that (for details see Section~\ref{sec:combine}) $\e(g(n)) \approx n^{it}$ on a long progression inside $(N, N+H]$ for certain $t$ with $|t| \leq (N/H)^{k+1}$. This means that, in order to handle the remaining case, it suffices to estimate the corresponding type II sums twisted by $n^{it}$ instead of $\e(g(n))$. Dirichlet polynomial techniques become applicable, and we shall use the following estimate which quickly follows from the work of Baker, Harman and Pintz~\cite{BHP01}.
\begin{lemma} 
\label{le:BHP}
Fix $\theta > 2/3$ and $A \geq 2$. Let $H = N^{\theta}$. Let $L, M_1 ,M_2 \geq 1$ be such that $M_j = N^{\alpha_j}$ and $L M_1 M_2 \asymp N$. Assume that $\alpha_1, \alpha_2 > 0$ obey the bounds
\begin{equation}
\label{eq:BHP-condition}
|\alpha_1 - \alpha_2| \leq \frac{1}{3} + \frac{\theta-2/3}{100} \quad \text{and} \quad 0 < |1 - \alpha_1 - \alpha_2| \leq \frac{4}{9}.
\end{equation}
Let $a_{m_1} ,b_{m_2} ,c_\ell$ be $\tau_5$-bounded coefficients. Suppose that
\begin{equation}
\label{eq:BHP-supt-cond}
\sup_{|t| \leq \frac{N}{H} (\log N)^{A+50}} \left| \sum_{\ell \sim L} \frac{c_\ell}{\ell^{1/2+it}}\right| \ll_C \frac{L^{1/2}}{(\log N)^{C}}
\end{equation}
for every $C > 0$. Then
\[
\left|\sum_{\substack{N < \ell m_1 m_2 \leq N+H \\ \ell \sim L, m_1 \sim M_1, m_2 \sim M_2}} a_{m_1} b_{m_2} c_\ell \right| \ll_A \frac{H}{(\log N)^A}.
\]
\end{lemma}

\begin{proof} When $T = N(\log N)^{A+50}/H$ and 
\[
F(s) = \sum_{\ell \sim L} \frac{c_\ell}{l^{s}} \sum_{m_1 \sim M_1} \frac{a_{m_1}}{m_1^{s}}\sum_{m_2 \sim M_2} \frac{b_{m_2}}{m_2^{s}},
\]
Perron's formula (see e.g. \cite[Corollary 5.3]{MontVau07}) together with the divisor bound (see~\eqref{eq:div-bound} below) implies
\[
\begin{split}
&\sum_{\substack{N < \ell m_1 m_2 \leq N+H \\ \ell \sim L, m_1 \sim M_1, m_2 \sim M_2}} a_{m_1} b_{m_2} c_\ell \\
&= \frac{1}{2\pi i} \int_{1/2-iT}^{1/2+iT} F(s) \frac{(N+H)^s - N^s}{s} \d s + O\left(\frac{H}{(\log N)^{A+2}}\right) \\
&\ll \min\left\{\frac{H}{N^{1/2}}, \frac{N^{1/2}}{T}\right\} \int_{-T}^{T} |F(1/2+it)| dt +  O\left(\frac{H}{(\log N)^{A+2}}\right).
\end{split}
\]
The claim now follows from~\cite[Lemma 9]{BHP01} with $g=1$ (alternatively see~\cite[Lemma 7.3]{Harman07}).
\end{proof}




We end this section by speculating on what happens with general nilsequences $\psi$.  We expect that the following rough statement can be proved by our minor arc argument, using the quantitative Leibman theorem due to Green and Tao~\cite[Theorem 2.9]{GT12} in place of Weyl's inequality. See~\cite{GT12} for the precise definitions of the terms below.

Let $G/\Gamma$ be a nilmanifold, $g$ be a polynomial sequence on $G$, and $\varphi$ be a smooth function on $G/\Gamma$. Let $\psi(n) = \varphi(g(n)\Gamma)$. If
\[ \left|\sum_{N < n \leq N+H} \Lambda(n) \psi(n) \right| \approx H, \]
then there is a nontrivial horizontal character $\chi$ on $G$ (with bounded modulus), such that the coefficients of the polynomial
\[ \chi \circ g(n) = \sum_{j=1}^k \alpha_j(n - N)^j \]
satisfy
\[ \|j\alpha_j + (j+1)N\alpha_{j+1}\| \leq \frac{N}{H^{j+1}} \]
for all $1 \leq j \leq k$. This is the same as saying that the polynomial $\chi\circ g(n)$ is roughly the same as $n^{it}$ on $(N, N+H]$, but we do not know how to use this information to say something about the nilsequence $\psi$.

\section{Proof of Propositions~\ref{prop:typeI} and \ref{prop:typeII}}
\label{sec:typeIandII}

 We need the following reformulation of Weyl's inequality. This is a direct consequence of~\cite[Proposition 4.3]{GT12}, and also a special case of a more general quantitative equidistribution result on nilsequences~\cite[Theorem 2.9]{GT12}.

\begin{lemma}\label{lem:weyl}
Let $N \geq 2$. Let $g(n) = \alpha_1 n + \cdots + \alpha_k n^k$ be a polynomial of degree $k$. If
\[ \left|\sum_{n \in I} \e(g(n))\right| \geq \delta N \]
for some interval $I \subset \{1,2,\cdots,N\}$ and some $\delta \in (0,1/2)$, then there exists a positive integer $q \leq \delta^{-O_k(1)}$ such that
\[ \|q\alpha_j\| \leq \delta^{-O_k(1)} \frac{1}{N^j} \]
for all $1 \leq j \leq k$.
\end{lemma}

Another useful lemma will be the following result of Green and Tao (see~\cite[Lemma 4.5]{GT12}).
\begin{lemma}
\label{GT-Lemma4.5}
Let $N \geq 2$ and let $g(n) = \alpha_1 n + \cdots + \alpha_k n^k$ be a polynomial of degree $k$. Suppose that $0 < \delta < 1/2$ and $\varepsilon \leq \delta/2$, that $I \subseteq \mathbb{R} / \mathbb{Z}$ is an interval of length $\varepsilon$ and that $g(n) \pmod{\mathbb{Z}} \in I$ for at least $\delta N$ values of $n \in \{1, \dotsc,  N\}$. Then there exists $q \in \mathbb{Z}$ with $0 < |q| \ll \delta^{-O_k(1)}$, such that $\| q \alpha_j \| \ll \varepsilon \delta^{-O_k(1)}/N^j$ for every $j = 1, \dotsc, k$.
\end{lemma}

We will constantly use the divisor bound (which follows e.g. from the Shiu bound~\cite{Shiu80}): For any $r, s \in \mathbb{N}$ and $\kappa > 0$,
\begin{equation}
\label{eq:div-bound}
\sum_{x < n \leq x + x^\kappa} \tau_r(n)^s \ll x^\kappa (\log x)^{r^s-1}.
\end{equation}

\begin{proof}[Proof of Proposition~\ref{prop:typeI}]
From the hypothesis we have
\[ \delta H (\log N)^{13} \leq \sum_{m \sim M} \tau_5(m) \left| \sum_{N/m < \ell \leq (N+H)/m} \psi(\ell) \e(g(\ell m)) \right|.  \]
By the Cauchy-Schwarz inequality and \eqref{eq:div-bound} this implies that
\begin{equation}\label{eq:Lemma3.2(1)} 
\delta^2 \frac{H^2}{M} (\log N)^{2} \ll \sum_{m \sim M} \left| \sum_{N/m < \ell \leq (N+H)/m} \psi(\ell) \e(g(\ell m)) \right|^2. 
\end{equation}

Denote by $\CM$ the set of $m \sim M$ such that
\[ \left| \sum_{N/m < \ell \leq (N+H)/m} \psi(\ell) \e(g(\ell m)) \right| \gg \delta \frac{H}{M} \log N. \]
Thus the contribution from those $m \notin \CM$ to the sum in~\eqref{eq:Lemma3.2(1)} is negligible compared to the lower bound, and it follows that
\[ \delta^2 \frac{H^2}{M} (\log N)^{2} \ll \sum_{m \in \CM} \left| \sum_{N/m < \ell \leq (N+H)/m} \psi(\ell) \e(g(\ell m)) \right|^2 \ll |\CM| \cdot \frac{H^2}{M^2}. \] 
Hence $|\CM| \gg \delta^2 M$. For $m \in \CM$, let $\ell_0 = \lceil N/m\rceil$ be the starting point of the range of summation over $\ell$. We can conclude that
\[ \left|\sum_{\ell_0 < \ell \leq \ell_1} \e(g(\ell m)) \right| \gg \delta \frac{H}{M} \]
for some $\ell_1 \leq (N+H)/m$. Indeed, when $\psi(\ell) = 1$ this is trivial, and when $\psi(\ell) = \log\ell$ this follows from partial summation. We will apply Lemma~\ref{lem:weyl} to the shifted sequence $\ell \mapsto g((\ell_0+\ell)m)$. Note that
\[ g((\ell_0+\ell)m) = \sum_{i=1}^k \alpha_i (\ell m + b)^i, \]
where $b = \ell_0m - N$. The only property we will use about $b$ is the bound $|b| \leq H$. The coefficient of $\ell^j$ in this polynomial is given by
\[ \beta_j := \sum_{j \leq i \leq k} \alpha_i \binom{i}{j} m^j b^{i-j}. \]
By Lemma~\ref{lem:weyl}, there exists a positive integer $q \leq \delta^{-O(1)}$ such that
\[ \|q \beta_j\| \leq \delta^{-O(1)} \left(\frac{M}{H}\right)^j \]
for all $1 \leq j \leq k$ and $m \in \CM$. Below we will allow ourselves to enlarge $q$ by multiplying it with a positive integer at most $\delta^{-O(1)}$, and this process will be done $O(1)$ times, so the bound $q \leq \delta^{-O(1)}$ will remain to hold in the end. We will show by induction the desired Diophantine information on $\alpha_j$:
\[ \|q \alpha_j\| \leq \delta^{-O(1)} \frac{1}{H^j} \]
for all $1 \leq j \leq k$.

\subsubsection*{The base case $j=k$} 

Since $\beta_k = \alpha_k m^k$, we have
\[ \| q\alpha_k m^k\| \leq \delta^{-O(1)} \left(\frac{M}{H}\right)^k. \]
This holds for $\gg \delta^2 M$ values of $m \sim M$. Hence by Lemma~\ref{GT-Lemma4.5} (which is applicable by our assumption that $M/H \leq \delta^C$), we conclude that
\[ \|q\alpha_k\| \leq \delta^{-O(1)} \frac{1}{H^k} \]
as desired.

\subsubsection*{The induction step}

Now let $1 \leq j < k$, and assume that the claim has already been proved for larger values of $j$. Then
\[ 
\begin{split}
\left\| q(\beta_j - \alpha_j m^j) \right\| &= \left\| q \sum_{j < i \leq k} \alpha_i \binom{i}{j} m^j b^{i-j} \right\| \\
&\ll M^j \sum_{j < i \leq k} |b|^{i-j} \|q\alpha_i\| \leq \delta^{-O(1)} \left(\frac{M}{H}\right)^j. 
\end{split}
\]
It follows that
\[ \|q \alpha_j m^j\| \leq \delta^{-O(1)} \left(\frac{M}{H}\right)^j. \]
This holds for $\gg \delta^2 M$ values of $m \sim M$. Hence by Lemma~\ref{GT-Lemma4.5} (which is again applicable by our assumption that $M/H \leq \delta^C$), we conclude that
\[ \|q \alpha_j\| \leq \delta^{-O(1)} \frac{1}{H^j}. \]
This completes the proof.
\end{proof}

\begin{proof}[Proof of Proposition~\ref{prop:typeII}]
From the hypothesis we have
\[ \delta H (\log N)^{32} \leq \sum_{L/2\leq \ell \leq 2L} \tau_5(\ell) \left|\sum_{\substack{m \sim M \\ N < \ell m \leq N+H}} b_m \e(g(\ell m)) \right|. \]
By the Cauchy-Schwarz inequality and \eqref{eq:div-bound}, we have
\[ \delta^2 H^2 (\log N)^{40} \ll L \sum_{L/2 \leq \ell \leq 2L} \left|\sum_{\substack{m \sim M \\ N < \ell m \leq N+H}} b_m \e(g(\ell m)) \right|^2. \]
Expanding the square and changing the order of summation, we obtain
\[ \sum_{\substack{m,m' \sim M \\ |m-m'| \leq 2H/L}} \tau_3(m) \tau_3(m') \left| \sum_{\substack{L/2 \leq \ell \leq 2L \\ N < \ell m, \ell m' \leq N+H}} \e( g(\ell m) - g(\ell m') ) \right| \gg \delta^2 \frac{H^2}{L} (\log N)^{40}. \]
By the Cauchy-Schwarz inequality and the bound
\[
\sum_{\substack{m,m' \sim M \\ |m-m'| \leq 2H/L}} (\tau_3(m) \tau_3(m'))^2 \ll \frac{H}{L} \sum_{m\sim M} \tau_3(m)^4 \ll \frac{HM}{L} (\log M)^{80}
\]
coming from~\eqref{eq:div-bound}, this implies that
\[ \sum_{\substack{m,m' \sim M \\ |m-m'| \leq 2H/L}} \left| \sum_{\substack{L/2 \leq \ell \leq 2L \\ N < \ell m, \ell m' \leq N+H}} \e( g(\ell m) - g(\ell m') ) \right|^2 \gg \delta^4 \frac{H^3}{LM}. \]

We will consider intervals of length $3H/L$ of the form $J = [m_0, m_0 + 3H/L]$ for some $M/2 \leq m_0 \leq 2M$. Since each pair $(m,m')$ with $m,m'\sim M$ and $|m-m'| \leq 2H/L$ appears in $\gg H/L$ such intervals $J$, we have
\[ \sum_{M/2 \leq m_0 \leq 2M} \sum_{m,m' \in [m_0,m_0 + 3H/L]}  \left| \sum_{\substack{L/2 \leq \ell \leq 2L \\ N < \ell m, \ell m' \leq N+H}} \e( g(\ell m) - g(\ell m') ) \right|^2 \gg \delta^4 \frac{H^4}{L^2M}. \]
Hence the inequality
\[ \sum_{m,m' \in J} \left| \sum_{\substack{L/2 \leq \ell \leq 2L \\ N < \ell m, \ell m' \leq N+H}} \e( g(\ell m) - g(\ell m') ) \right|^2 \gg \delta^4 \left(\frac{H}{L}\right)^2 \left(\frac{H}{M}\right)^2 \]
holds for $\gg \delta^4 M$ choices of $m_0$. For the moment we fix one such choice of $m_0$ and $J$, but towards the end of the argument we will allow $m_0$ to vary. Denote by $\CM$ the set of all pairs $(m,m') \in J \times J$ such that
\[ \left|\sum_{\substack{L/2 \leq \ell \leq 2L \\ N < \ell m, \ell m' \leq N+H}} \e(g(\ell m) - g(\ell m')) \right| \gg \delta^2 \frac{H}{M}. \]
It follows that $|\CM| \gg \delta^4 |J|^2$. For $(m,m') \in \CM$, let $I_{m,m'}$ be the range of summation for $\ell$:
\[ I_{m,m'} = \{L/2 \leq \ell \leq 2L : N < \ell m, \ell m' \leq N+H\}. \]
Note that all of these $I_{m,m'}$ are contained in a common interval $I = [\ell_0+1, \ell_0+|I|]$ of length $|I| = O(H/M)$ for some $\ell_0 = N/m_0 + O(H/M)$, which depends on $m_0$ but not on $m,m'$. We are now in a position to apply Lemma~\ref{lem:weyl} to the shifted sequence
\[ \ell \mapsto g((\ell_0+\ell)m) - g((\ell_0+\ell)m'). \]
Note that
\[ g((\ell_0+\ell) m) - g((\ell_0+\ell) m') =  \sum_{i=1}^k \alpha_i [(\ell m + b)^i - (\ell m' + b')^i], \]
where $b = \ell_0m-N$, $b' = \ell_0m' - N$  (bear in mind the dependence of $b,b'$ on $m,m'$). The coefficient of $\ell^j$ in this polynomial is given by
\[ \beta_j(m,m') := \sum_{j \leq i \leq k} \alpha_i \binom{i}{j} (m^j b^{i-j} - m'^j b'^{i-j}). \]
By Lemma~\ref{lem:weyl}, there exists a positive integer $q \leq \delta^{-O(1)}$ such that
\[ \|q \beta_j(m,m')\| \leq \delta^{-O(1)} \frac{1}{|I|^j} \leq \delta^{-O(1)} \left(\frac{M}{H}\right)^j \]
for all $1 \leq j \leq k$ and $(m,m') \in \CM$. In the rest of the arguments we will always allow ourselves to enlarge $q$ by multiplying it with a positive integer at most $\delta^{-O(1)}$, and this process will be done $O(1)$ times so that the bound $q \leq \delta^{-O(1)}$ will remain to hold in the end.  Let
\[ \gamma_j(m) = m^j \sum_{j \leq i \leq k} \alpha_i \binom{i}{j} (\ell_0m - N)^{i-j}, \]
so that $\beta_j(m,m') = \gamma_j(m) - \gamma_j(m')$. The Diophantine information on $\beta_j(m,m')$ implies that $\|q\gamma_j(m)\|$ lies in an arc of length $\delta^{-O(1)} (M/H)^j$ for $\gg \delta^4 |J|$ values of $m \in J$, and thus $\|q\gamma_j(m_0+m)\|$ lies in an arc of length $\delta^{-O(1)} (M/H)^j$ for $\gg \delta^4 H/L$ values of  $1 \leq m \leq 3H/L$. Using this we will obtain the desired Diophantine information:
\[ \|q (j \alpha_j + (j+1)N \alpha_{j+1})\| \leq \delta^{-O(1)} \frac{N}{H^{j+1}} \]
by induction on $j$.

\subsubsection*{The base case $j=k$}

Note that
\[ \gamma_k(m_0+m) = \alpha_k (m_0+m)^k. \]
As a polynomial in $m$, its linear coefficient is $km_0^{k-1}\alpha_k$. By Lemma~\ref{GT-Lemma4.5} (which is applicable by our assumption that $M/H \leq \delta^C$ for some sufficiently large $C$), we deduce that
\[ \|q k m_0^{k-1} \alpha_k\| \leq \delta^{-O(1)} \left(\frac{M}{H}\right)^k \frac{L}{H}. \]
Recall that this holds for $\gg \delta^4 M$ values of  $M/2 \leq m_0 \leq 2M$, so by Lemma~\ref{GT-Lemma4.5} again (which is again applicable by our assumption that $M/H, L/H \leq \delta^C$), we conclude that
\[ \|qk \alpha_k\| \leq \delta^{-O(1)} \frac{ML}{H^{k+1}} = \delta^{-O(1)} \frac{N}{H^{k+1}}, \]
as desired. 

\subsubsection*{The induction step}

Now let $1 \leq j < k$, and assume that the claim has already been proved for larger values of $j$. Note that
\[ \gamma_j(m_0+m) = (m_0+m)^j \sum_{j \leq i \leq k} \alpha_i \binom{i}{j} (\ell_0m + h_0)^{i-j},
\]
where $h_0 = \ell_0m_0 - N$ satisfies $|h_0| = O(H)$. As a polynomial in $m$, its linear coefficient is
\[ \lambda = m_0^j \sum_{j < i \leq k} \alpha_i \binom{i}{j} (i-j) \ell_0 h_0^{i-j-1} + j m_0^{j-1} \sum_{j \leq i \leq k} \alpha_i \binom{i}{j} h_0^{i-j}, \]
and by Lemma~\ref{GT-Lemma4.5} we have
\begin{equation}
\label{eq:qlambdabound}
\|q \lambda\| \leq \delta^{-O(1)} \left(\frac{M}{H}\right)^j \frac{L}{H}. 
\end{equation}
The expression for $\lambda$ can be rewritten as
\[ \lambda = m_0^{j-1} \sum_{j \leq i \leq k} \alpha_i \binom{i}{j}( (i-j) \ell_0m_0h_0^{i-j-1} + j h_0^{i-j}). \]
Since $\ell_0m_0 = h_0+N$, we have
\[ \lambda = m_0^{j-1} \sum_{j\leq i \leq k} \alpha_i \binom{i}{j} ((i-j)N h_0^{i-j-1} + i h_0^{i-j}). \]
By regrouping terms according to the exponent of $h_0$ we obtain
\[ 
\begin{split}
\lambda &= m_0^{j-1} \sum_{j \leq i \leq k} h_0^{i-j}\left(i \alpha_i \binom{i}{j} + (i+1-j)N \alpha_{i+1} \binom{i+1}{j} \right) \\
&= m_0^{j-1} \sum_{j \leq i \leq k} h_0^{i-j} \binom{i}{j} (i\alpha_i + (i+1)N\alpha_{i+1}). \\
\end{split}
\]
By induction hypothesis, we know that, when considering $\|q \lambda\|$, all summands above with $i > j$ contribute
\[ \ll H^{i-j} \delta^{-O(1)} \frac{N}{H^{i+1}} \leq \delta^{-O(1)} \frac{N}{H^{j+1}},  \]
and thus
\[ \|q \lambda\| = \|q m_0^{j-1}  (j\alpha_j + (j+1)N\alpha_{j+1})\| + O\left(\delta^{-O(1)} \frac{M^{j-1}N}{H^{j+1}} \right). \]
Hence the bound~\eqref{eq:qlambdabound} on $\|q\lambda\|$ implies that
\[ \|q m_0^{j-1} (j\alpha_j + (j+1)N\alpha_{j+1})\| \leq \delta^{-O(1)} \frac{M^{j-1}N}{H^{j+1}}. \]
Since this is true for $\gg \delta^4 M$ values of $M/2 \leq m_0 \leq 2M$, another application of Lemma~\ref{GT-Lemma4.5}  leads to
\[ \|q (j\alpha_j + (j+1)N\alpha_{j+1})\| \leq \delta^{-O(1)} \frac{N}{H^{j+1}}, \]
as desired. This completes the proof.
\end{proof}

\section{Proof of Theorems~\ref{thm:exp-alpha},~\ref{thm:exp-g}, and~\ref{thm:exp-g-mu}}
\label{sec:combine}

\begin{proof}[Proof of Theorem~\ref{thm:exp-g}] Assume that
\[ 
\left|\sum_{N < n \leq N +H} \Lambda(n) \e(g(n))\right| \geq \frac{H}{(\log N)^A}.
\]
By Heath-Brown's identity~\cite[Section 2]{H-B82} (alternatively see e.g.~\cite[Section 2.5]{Harman07}), the left hand side equals
\[
\sum_{j = 1}^3 (-1)^{j-1} \binom{3}{j} \sum_{\substack{N < r_1 \dotsm r_{2j} \leq N +H \\ i > j \implies r_i \leq (2N)^{1/3}}} (\log r_1) \mu(r_{j+1}) \dotsm \mu(r_{2j}) \e(g(r_1 \dotsm r_{2j}))
\]
Splitting the summation variables into dyadic ranges, we see that, for each appearing sum, at least one of the following three cases must occur:
\begin{enumerate}
\item (Type-I) For some $M \leq N^{2/3}$ and $|b_m| \leq \tau_5(m)$, we have
\[ \sum_{\substack{\ell, m \\ m \sim M \\ N < \ell m \leq N+H}} b_m \e(g(\ell m)) \gg \frac{H}{(\log N)^{A+6}}. \]
\item (Type-I) For some $M \leq N^{2/3}$ and $|b_m| \leq \tau_5(m)$, we have
\[ \sum_{\substack{\ell, m \\ m \sim M \\ N < \ell m \leq N+H}} (\log\ell) b_m \e(g(\ell m)) \gg \frac{H}{(\log N)^{A+6}}. \]
\item (Type-II) For some $R_1, \dotsc, R_6 \leq 2N^{1/3}$ with $R_1 \dotsm R_6 \asymp N$, we have
\[ \sum_{\substack{r_1, \dotsc, r_6 \\ r_i \sim R_i \\ N < r_1 \dotsm r_6 \leq N+H}} a^{(1)}_{r_1} \dotsm a^{(6)}_{r_6} \e(g(r_1 \dotsm r_6)) \gg \frac{H}{(\log N)^{A+6}}, \]
where each sequence $a^{(j)}_r$ is one of $1$, $\log r$ or $\mu(r)$.
\end{enumerate}
Indeed, if $r_1, \dotsc, r_j \leq 2N^{1/3}$ then the term belongs to case (3), and if $r_i > 2N^{1/3}$ for some $i \in \{1, \dotsc ,j\}$, we may write $\ell = r_i$ and end up in case (1) if $i \neq 1$ or case (2) if $i = 1$.

In either case (1) or (2), the claim follows immediately from Proposition~\ref{prop:typeI} with $\delta = (\log N)^{-A - 20}$. In case (3) we first notice that some product of $R_i$ must lie in $[N^{1/3}, 2N^{2/3}]$. Hence we can apply Proposition~\ref{prop:typeII} with $\delta = (\log N)^{-A-43}$ to find a positive integer $q \leq (\log N)^{O(A)}$ such that
\[ \|q(j\alpha_j + (j+1)N\alpha_{j+1})\| \leq (\log N)^{O(A)} \frac{N}{H^{j+1}} \]
for all $1 \leq j\leq k$, with the convention that $\alpha_{k+1} = 0$. 

Let $B = CA$ for some sufficiently large constant $C = C(k)$, and let $H_0 = H(\log N)^{-B}$. We divide $(N, N+H]$ into arithmetic progressions of the form
\[ P = \{n_0 < n \leq n_0 + H_0 \colon n \equiv a \pmod{k!q}\}, \]
where $n_0 \in [N, N+H]$ and $(a, k!q)=1$. Our hypothesis implies that for at least one such progression $P$, we have
\begin{equation}
\label{eq:ri-sum-P-large} 
\left|\sum_{\substack{r_1 \dotsm r_6 \in P \\ r_i \sim R_i}} a^{(1)}_{r_1} \dotsm a^{(6)}_{r_6} \e(g(r_1 \dotsm r_6)) \right| \gg \frac{|P|}{(\log N)^{A+6}}.  
\end{equation}
For the remainder of the proof we fix such a progression $P$. We claim that there exists some $\eta$ with $|\eta| = 1$ and some $t$ with $|t| \leq (N/H)^{k+1}(\log N)^{O(A)}$, such that
\begin{equation}\label{eq:g-nit} 
\e(g(n)) = \eta n^{it} (1 + O((\log N)^{-A-15})) 
\end{equation}
for all $n \in P$. To see this, first write
\[ g(n) = \sum_{j=1}^k \alpha_j(n - N)^j = \sum_{j=1}^k \beta_j (n-n_0)^j, \]
so that
\[ \beta_j = \sum_{i=j}^k \binom{i}{j} (n_0 - N)^{i-j} \alpha_i. \]
After some algebra one derives that
\[ j\beta_j + (j+1) n_0\beta_{j+1} = \sum_{i=j}^k \binom{i}{j} (n_0-N)^{i-j} [i\alpha_i + (i+1) N\alpha_{i+1}] \]
for all $1 \leq j \leq k$, with the convention that $\beta_{k+1} = 0$. Hence
\[ \|q(j\beta_j + (j+1)n_0\beta_{j+1}\| \ll \sum_{i=j}^k H^{i-j} (\log N)^{O(A)} \frac{N}{H^{i+1}} \leq (\log N)^{O(A)} \frac{N}{H^{j+1}}. \]
Now shift each $\beta_j$ by $(qj)^{-1}a_j$ for an appropriate $a_j \in \Z$ to get $\beta_j'$, so that
\begin{equation}\label{eq:alphak} 
|q(j\beta_j' + (j+1)n_0\beta_{j+1}')| \leq (\log N)^{O(A)} \frac{N}{H^{j+1}}
\end{equation}
for all $1 \leq j \leq k$. Let
\[ g'(n) = \sum_{j=1}^k \beta_j'(n-n_0)^j. \]
Note that for $n \in P$ we have
\[ \e(g(n)) = \e(g'(n)) \e\left(\sum_{j=1}^k \frac{a_j}{qj} (n-n_0)^j\right) = \eta\e(g'(n)),  \]
for some $\eta$ (independent of $n$) with $|\eta|=1$, since all $n \in P$ lie in the same residue class modulo $qj$ for each $j$.  By  induction one can deduce from~\eqref{eq:alphak} that
\begin{equation}\label{eq:alphak'} 
\left|\beta_j' - \frac{(-1)^{j-1}}{jn_0^{j-1}} \beta_1'\right| \leq (\log N)^{O(A)} \frac{1}{H^j}
\end{equation}
for all $1 \leq j \leq k  + 1$. In particular when $j=k+1$ this gives 
\[ |\beta_1'| \leq (\log N)^{O(A)} \frac{N^k}{H^{k+1}}. \]
Set $t = 2\pi n_0\beta_1'$, so that
\[ |t| \leq (\log N)^{O(A)} \left(\frac{N}{H}\right)^{k+1}. \]
For $n \in P$ we have
\[ n^{it} = n_0^{it} \exp\left( it\log\left(1 + \frac{n-n_0}{n_0}\right) \right). \]
Using the Taylor expansion
\[ \log\left(1 + \frac{n-n_0}{n_0}\right) = \sum_{j=1}^k \frac{(-1)^{j-1}}{j} \left(\frac{n-n_0}{n_0}\right)^j + O\left( \left(\frac{|n-n_0|}{N}\right)^{k+1} \right),
\]
we get
\[ n^{it} = n_0^{it} \e( \wt{g}(n)) \left(1 + O\left( |t|\left(\frac{H_0}{N}\right)^{k+1} \right) \right), \]
where 
\[ \wt{g}(n) = \frac{t}{2\pi} \sum_{j=1}^k \frac{(-1)^{j-1}}{j} \left(\frac{n-n_0}{n_0}\right)^j = \sum_{j=1}^k \frac{(-1)^{j-1}}{j n_0^{j-1}} \beta_1' (n-n_0)^j. \]
The error term above can be made $O((\log N)^{-A-15})$ by  choosing $B$ in the definition of $H_0$ large enough. Hence
\[ n^{it} = n_0^{it} \e(\wt{g}(n)) (1 + O( (\log N)^{-A-15} )). \]
Note that for $n \in P$ we have
\[ |g'(n) - \wt{g}(n)| \leq \sum_{j=1}^k \left|\beta_j' - \frac{(-1)^{j-1}}{jn_0^{j-1}} \beta_1'\right| (n-n_0)^j \leq \sum_{j=1}^k (\log N)^{O(A)} \frac{H_0^j}{H^j}, \]
which can again be made $\leq (\log N)^{-A-15}$ by choosing $B$ large enough. It follows that
\[ n^{it} = n_0^{it} \e(g'(n)) (1 + O( (\log N)^{-A-15} )) \]
for $n \in P$, and hence
\[ \e(g(n)) = (\eta n_0^{-it}) n^{it} (1 + O( (\log N)^{-A-15} )). \]
This establishes the claim~\eqref{eq:g-nit}. 

It then follows from~\eqref{eq:ri-sum-P-large} that
\[ \left|\sum_{\substack{r_1 \dotsm r_6 \in P \\ r_i \sim R_i}} a^{(1)}_{r_1} \dotsm a^{(6)}_{r_6} (r_1 \dotsm r_6)^{it} \right| \gg \frac{|P|}{(\log N)^{A+6}}.  \]
Decomposing this sum using Dirichlet characters mod $k!q$, we get
\begin{equation}
\label{eq:largeaisum} 
\left|\sum_{\substack{n_0 < r_1 \dotsm r_6 \leq n_0 + H_0 \\ r_i \sim R_i}} a^{(1)}_{r_1} \dotsm a^{(6)}_{r_6} \chi(r_1 \dotsm r_6) (r_1 \dotsm r_6)^{it}\right| \gg \frac{H_0}{q(\log N)^{A+6}} 
\end{equation}
for some $\chi\pmod{k!q}$. We can ensure that the modulus $k!q \leq (\log N)^B$, and moreover the lower bound above is $\gg H_0 (\log N)^{-A-B-6}$.  Now we want to apply Lemma~\ref{le:BHP} with $n_0$ in place of $N$, $H_0$ in place of $H$, and $A+B+7$ in place of $A$. Since $R_i \leq 2N^{1/3}$ for each $i$, we can arrange $R_i$ into three factors such that~\eqref{eq:BHP-condition} holds (e.g. take $L$ to be some product of $R_i$ for which $L \in [n_0^{1/9}/4, n_0^{4/9}]$ and then arrange the remianing $R_i$ into products $M_1$ and $M_2$ for which $M_1/M_2$ and $M_2/M_1$ are at most $2N^{1/3}$). The coefficients $a_{m_1}, b_{m_2}$ and $c_l$ in application of Lemma~\ref{le:BHP} are then taken to be convolutions of the corresponding sequences $a^{(i)}_{r_i} \chi(r_i) r_i^{it}$. 

Recalling the special shape of $a^{(j)}_r$, van der Corput exponential sum estimates and the zero-free region for $L(s, \chi)$ imply that \eqref{eq:BHP-supt-cond} holds unless $|t| \leq 2 \frac{n_0}{H_0} (\log N)^{A+B+57}$ and $\chi$ is principal (this follows directly e.g. from~\cite[Lemma 2.7]{MRT19}). Lemma~\ref{le:BHP} is thus applicable if $|t| > 2 \frac{n_0}{H_0} (\log N)^{A+B+57}$. But the conclusion of Lemma~\ref{le:BHP} contradicts with~(\ref{eq:largeaisum}), so we can conclude that
\[ |t| \leq 2\frac{n_0}{H_0}(\log N)^{A+B+57} \leq \frac{N}{H}(\log N)^{4B}. \]
By the definition of $t$, it follows that
\[ |\beta_1'| \leq \frac{(\log N)^{4B}}{H},  \]
and then by~\eqref{eq:alphak'} we get
\[ |\beta_j'| \leq \frac{(\log N)^{4B}}{N^{j-1}H} + \frac{(\log N)^{O(A)}}{H^j} \leq \frac{(\log N)^{O(A)}}{H^j}. \]
Hence 
\[ \|(k!q)\beta_j\| = \|(k!q)\beta_j'\| \leq \frac{(\log N)^{O(A)}}{H^j}.  \]
Finally, using the relation
\[ \alpha_j = \sum_{i=j}^k \binom{i}{j} (N - n_0)^{i-j} \beta_i, \]
one arrives at the desired inequality
\[ \|(k!q)\alpha_j\| \leq \sum_{i=j}^k \binom{i}{j} H^{i-j} \frac{(\log N)^{O(A)}}{H^i} \leq \frac{(\log N)^{O(A)}}{H^j}.   \]
\end{proof}

\begin{proof}[Proof of Theorem~\ref{thm:exp-g-mu}]
The argument is the same as above, except that we start with a variant of Heath-Brown's identity for $\mu$ (which can be obtained from~\cite[Lemma 1]{H-B82} by dividing both sides by $\zeta'(s)$ and comparing coefficients), which leads to
\[ \sum_{N < n \leq N+H} \mu(n) \e(g(n)) = \sum_{j=2}^3 (-1)^{j-1} \binom{3}{j} \sum_{\substack{N < r_2 \cdots r_{2j} \leq N+H \\ i > j \implies r_i \leq (2N)^{1/3}}} \mu(r_{j+1}) \cdots \mu(r_{2j}) \e(g(r_2\cdots r_{2j})).  \]
At the end of the argument, it remains to treat the major arc case when the coefficients from
\[ g(n) = \sum_{j=1}^k \alpha_j (n-N)^j \]
satisfy the conditions
\[ \|q\alpha_j\| \leq \frac{(\log N)^{O_k(A)}}{H^j} \]
for all $1 \leq j \leq k$. After dividing $(N, N+H]$ into subprogressions, this easily follows from known bounds for
\[ \sum_{N < n \leq N + H(\log N)^{-O_k(A)}} \mu(n) \chi(n), \]
where $\chi$ is a Dirichlet character with modulus $\leq (\log N)^{O_k(A)}$ (analogously to primes in short intervals (see e.g. \cite[Section 10.5]{IK04}), this expression can be satisfactorily bounded using zero-density estimates for $L(s, \chi)$ when $\theta > 7/12$).
\end{proof}



\begin{proof}[Proof of Theorem~\ref{thm:exp-alpha}]
If we write
\[ g(n) = \alpha n^k = \sum_{j=1}^k \alpha_j (n-N)^j, \]
then $\alpha_j = \binom{k}{j} N^{k-j}\alpha$. Hence Theorem~\ref{thm:exp-g} implies that there exists a positive integer $q \leq (\log N)^{O_k(A)}$ such that
\begin{equation}\label{eq:alpha-k-thm} 
\left\|q \binom{k}{j} N^{k-j} \alpha \right\| \leq \frac{(\log N)^{O_k(A)}}{H^{j}} 
\end{equation}
for all $1 \leq j \leq k$. Let $q'$ be the least common multiple of $q \binom{k}{j}$ ($1 \leq j \leq k$), so that $q' \leq (\log N)^{O_k(A)}$. We will show by induction on $j$ that
\[ \left\| q'\alpha \right\| \leq \frac{(\log N)^{O_k(A)}}{N^{k-j} H^{j}} \]
for all $1 \leq j \leq k$, and the conclusion follows from the $j=1$ case of this. When $j = k$, the claim follows from~\eqref{eq:alpha-k-thm} with $j=k$. Now let $1 \leq j < k$, and assume that the claim has already been proven for $j+1$. The induction hypothesis implies that
\[ N^{k-j} \|q' \alpha\| \leq (\log N)^{O_k(A)} \frac{N^{k-j}}{N^{k-1-j}H^{j+1}} \leq (\log N)^{O_k(A)} \frac{N}{H^{j+1}} < \frac{1}{2}. \]
Combining this with
\[ \|N^{k-j} q'\alpha\| \leq \frac{(\log N)^{O_k(A)}}{H^{j}} \]
from~\eqref{eq:alpha-k-thm} leads to
\[ N^{k-j} \|q'\alpha\| \leq  \frac{(\log N)^{O_k(A)}}{H^{j}}, \]
which completes the induction step.
\end{proof}


\section{Application to the Waring-Goldbach problem}\label{sec:app}

Now that we are equipped with the exponential sum estimate Theorem~\ref{thm:exp-alpha}, we can deduce Theorem~\ref{thm:waring-goldbach} via the circle method. In this section we sketch this standard deduction. Let $X = (N/s)^{1/k}$, $H = X^{\theta}$, and let
\[ f(\alpha) = \sum_{|n-X| \leq H} \Lambda(n) \e(\alpha n^k). \]
Then the (weighted) number of ways to write 
\[ N = p_1^k + \cdots + p_s^k \]
with $p_1,\cdots,p_s$ primes satisfying $|p_i - X| \leq X^{\theta}$ is
\[ \rho(N) = \int_0^1 f(\alpha)^s \e(-N\alpha) \d\alpha. \]
Set $Q = (\log N)^A$ for a sufficiently large constant $A$. For $1 \leq a \leq q \leq Q$ with $(a,q)=1$, define
\[ \major(q,a) = \left\{\alpha \in [0,1] \colon |q\alpha - a| \leq \frac{Q}{X^{k-1}H} \right\}.  \]
Let $\major$ be the union of all such $\major(q,a)$, and let $\minor$ be the complement $[1/(X^{k-1}H), 1+1/(X^{k-1}H)]\setminus\major$. We caution that our definition of $\major$ here consists only of the genuine major arcs, while the definition of $\major$ in~\cite[Section 2]{WW15} consists also of the wide major arcs. We have $\rho(N) = \rho(N; \major) + \rho(N; \minor)$, where
\[ \rho(N;\major) =  \int_{\major} f(\alpha)^s \e(-N\alpha) \d\alpha, \ \  \rho(N;\minor) = \int_{\minor} f(\alpha)^s \e(-N\alpha) \d\alpha. \]
Theorem~\ref{thm:waring-goldbach} follows once we show that
\[ \rho(N; \major) \gg \frac{H^{s-1}}{X^{k-1}}, \ \ \rho(N; \minor) = o\left(\frac{H^{s-1}}{X^{k-1}}\right). \]

\subsubsection*{Analysis of $\rho(N; \major)$}

The width of our major arc is chosen so that if $\alpha \in \major(q,a)$, then $f(\alpha)$ can be estimated by counting primes in short intervals in residue classes modulo $q$. Since $\theta > 7/12$, we may use Huxley's result on primes in short intervals~\cite{Hux72} to get
\[ f\left(\frac{a}{q} + \beta\right) = \varphi(q)^{-1} S(q,a)  v(\beta) + O\left(\frac{H}{(\log X)^{10}}\right)  \]
for $|\beta| \leq Q/(X^{k-1}H)$, where
\[ S(q,a) = \sum_{\substack{1 \leq b \leq q \\ (b,q) = 1}} \e\left(\frac{ab^k}{q}\right), \]
and 
\[ v(\beta) = k^{-1} \sum_{(X-H)^k \leq m \leq (X+H)^k} m^{-1+1/k} \e(\beta m). \]
From this point on, the standard theory of the major arc contributions in the Waring-Goldbach problem  can be applied to yield the estimate
\[ 
\rho(N;\major) =  \FG(N) \FJ(N) + O\left(\frac{H^{s-1}}{X^{k-1} (\log X)^{10}}\right), 
\]
where $\FG(N)$ is the singular series
\[ \FG(N) = \sum_{q=1}^{\infty} \varphi(q)^{-s} \sum_{\substack{1 \leq a \leq q \\ (a,q)=1}} S(q,a)^s \e(-aN/q), \]
and $\FJ(N)$ is the singular integral
\[ \FJ(N) = \int_0^1 v(\beta)^s \e(-\beta N) \d \beta. \]
See~\cite[Section 2]{WW15} and the references therein. Moreover, under the assumption on $s$ and the congruence condition on $N$, it can be shown that
\[ \FG(N) \asymp 1, \ \ \FJ(N) \asymp \frac{H^{s-1}}{X^{k-1}}. \]
Hence $\rho(N; \major) \gg H^{s-1}/X^{k-1}$ as desired. 

\subsubsection*{Analysis of $\rho(N; \minor)$}

Let $t = k(k+1)/2+1$ and choose $B > 2t/(s-2t)$. Since $A$ can be chosen sufficiently large in terms of $B$, Theorem~\ref{thm:exp-alpha} implies that $|f(\alpha)| \leq H(\log X)^{-B}$ for $\alpha \in \minor$. Thus
\[ \rho(N; \minor) \ll \left(\frac{H}{(\log X)^B}\right)^{s-2t} \int_0^1 |f(\alpha)|^{2t} \d\alpha. \]
It suffices to establish the following mean value estimate:
\[ \int_0^1 |f(\alpha)|^{2t} \ll \frac{H^{2t-1}}{X^{k-1}} (\log X)^{2t}. \]
This is basically~\cite[Proposition 2.2]{WW15}; we just need to apply the Vinogradov mean value theorem without the $X^{\ee}$ loss. Let
\[ F(\alpha) = \sum_{|n-X| \leq H} \e(\alpha n^k). \]
By considering the underlying Diophantine equations, we get
\[ \int_0^1 |f(\alpha)|^{2t} \d\alpha \ll (\log X)^{2t} \int_0^1 |F(\alpha)|^{2t}\d\alpha. \]
An argument of Daemen~\cite{Dae10} (see~\cite[Lemma 3.1]{WW15}) shows that
\[ \int_0^1 |F(\alpha)|^{2t} \d\alpha \ll \frac{H^{k(k+1)/2-1}}{X^{k-1}} J_{t,k}(H), \]
where $J_{t,k}(H)$ is the number of integral solutions to the system of Diophantine equations
\[ x_1^j + \cdots + x_t^j = y_1^j + \cdots + y_t^j, \ \ 1 \leq j \leq k, \]
with $1 \leq x_1,\cdots,x_t,y_1,\cdots,y_t \leq H$. The Vinogradov mean value conjecture (see~\cite[Section 5]{BDG16}) gives that
\[ J_{t,k}(H) \ll H^{2t - k(k+1)/2} \]
for $t > k(k+1)/2$. Combining the inequalities above together, we get
\[ \int_0^1 |f(\alpha)|^{2t} \ll (\log X)^{2t} \frac{H^{2t-1}}{X^{k-1}}. \]
Hence $\rho(N; \minor) = o(H^{s-1}/X^{k-1})$ by our choice of $B$.

\bibliographystyle{plain}
\bibliography{biblio}{}

\begin{thebibliography}{10}

\bibitem{BHP01}
R.~C. Baker, G.~Harman, and J.~Pintz.
\newblock The difference between consecutive primes. {II}.
\newblock {\em Proc. London Math. Soc. (3)}, 83(3):532--562, 2001.

\bibitem{BDG16}
J.~Bourgain, C.~Demeter, and L.~Guth.
\newblock Proof of the main conjecture in {V}inogradov's mean value theorem for
  degrees higher than three.
\newblock {\em Ann. of Math. (2)}, 184(2):633--682, 2016.

\bibitem{Dae10}
D.~Daemen.
\newblock The asymptotic formula for localized solutions in {W}aring's problem
  and approximations to {W}eyl sums.
\newblock {\em Bull. Lond. Math. Soc.}, 42(1):75--82, 2010.

\bibitem{GT10}
B.~Green and T.~Tao.
\newblock Linear equations in primes.
\newblock {\em Ann. of Math. (2)}, 171(3):1753--1850, 2010.

\bibitem{GT12mobius}
B.~Green and T.~Tao.
\newblock The {M}\"{o}bius function is strongly orthogonal to nilsequences.
\newblock {\em Ann. of Math. (2)}, 175(2):541--566, 2012.

\bibitem{GT12}
B.~Green and T.~Tao.
\newblock The quantitative behaviour of polynomial orbits on nilmanifolds.
\newblock {\em Ann. of Math. (2)}, 175(2):465--540, 2012.

\bibitem{Harman07}
G.~Harman.
\newblock {\em Prime-detecting sieves}, volume~33 of {\em London Mathematical
  Society Monographs Series}.
\newblock Princeton University Press, Princeton, NJ, 2007.

\bibitem{H-B82}
D.~R. Heath-Brown.
\newblock Prime numbers in short intervals and a generalized {V}aughan
  identity.
\newblock {\em Canad. J. Math.}, 34(6):1365--1377, 1982.

\bibitem{Hua15}
B.~Huang.
\newblock Strong orthogonality between the {M}\"{o}bius function and nonlinear
  exponential functions in short intervals.
\newblock {\em Int. Math. Res. Not. IMRN}, (23):12713--12736, 2015.

\bibitem{Hua16}
B.~Huang.
\newblock Exponential sums over primes in short intervals and an application to
  the {W}aring-{G}oldbach problem.
\newblock {\em Mathematika}, 62(2):508--523, 2016.

\bibitem{Hux72}
M.~N. Huxley.
\newblock On the difference between consecutive primes.
\newblock {\em Invent. Math.}, 15:164--170, 1972.

\bibitem{IK04}
H.~Iwaniec and E.~Kowalski.
\newblock {\em Analytic number theory}, volume~53 of {\em American Mathematical
  Society Colloquium Publications}.
\newblock American Mathematical Society, Providence, RI, 2004.

\bibitem{MMS17}
K.~Matom\"{a}ki, J.~Maynard, and X.~Shao.
\newblock Vinogradov's theorem with almost equal summands.
\newblock {\em Proc. Lond. Math. Soc. (3)}, 115(2):323--347, 2017.

\bibitem{MR16}
K.~Matom\"{a}ki and M.~Radziwi\l\l.
\newblock Multiplicative functions in short intervals.
\newblock {\em Ann. of Math. (2)}, 183(3):1015--1056, 2016.

\bibitem{MRT19}
K.~Matom{\"a}ki, M.~Radziwi{\l}{\l}, and T.~Tao.
\newblock Correlations of the von {M}angoldt and higher divisor functions {I}.
  {L}ong shift ranges.
\newblock {\em Proc. Lond. Math. Soc. (3)}, 118:284--350, 2019.

\bibitem{MRT18}
K.~Matom{\"a}ki, M.~Radziwi{\l}{\l}, and T.~Tao.
\newblock Fourier uniformity of bounded multiplicative functions in short
  intervals on average, arXiv:1812.01224.

\bibitem{MontVau07}
H.~L. Montgomery and R.~C. Vaughan.
\newblock {\em Multiplicative number theory. {I}. {C}lassical theory},
  volume~97 of {\em Cambridge Studies in Advanced Mathematics}.
\newblock Cambridge University Press, Cambridge, 2007.

\bibitem{Salmensuu19}
J.~Salmensuu.
\newblock On the {W}aring-{G}oldbach problem with almost equal summands,
  arXiv:1903.01824.

\bibitem{Shiu80}
P.~Shiu.
\newblock A {B}run-{T}itchmarsh theorem for multiplicative functions.
\newblock {\em J. Reine Angew. Math.}, 313:161--170, 1980.

\bibitem{WW15}
B.~Wei and T.~D. Wooley.
\newblock On sums of powers of almost equal primes.
\newblock {\em Proc. Lond. Math. Soc. (3)}, 111(5):1130--1162, 2015.

\bibitem{Zha91}
T.~Zhan.
\newblock On the representation of large odd integer as a sum of three almost
  equal primes.
\newblock {\em Acta Math. Sinica (N.S.)}, 7(3):259--272, 1991.
\newblock A Chinese summary appears in Acta Math. Sinica {{\bf{3}}5} (1992),
  no. 4, 575.

\end{thebibliography}

\end{document}